\documentclass[11pt]{article}

\usepackage{amsfonts, amsmath, amsthm, latexsym, color, epsfig}
\usepackage[colorlinks=true,citecolor=black,linkcolor=black,urlcolor=blue]{hyperref}

\theoremstyle{plain}
\newtheorem{THM}{Theorem}[section]
\newtheorem*{THM*}{Theorem}

\newtheorem{CLAIM}[THM]{Claim}

\theoremstyle{definition}

\newtheorem*{HJ}{Hales--Jewett theorem}

\theoremstyle{definition}

\newtheorem{QUE}[THM]{Question}

\theoremstyle{remark}

\title{\vspace{-22pt} Intervals in the Hales--Jewett theorem}
 \author{David Conlon\thanks{Mathematical Institute, Oxford OX2 6GG, United Kingdom. Email: \href{david.conlon@maths.ox.ac.uk} {\nolinkurl{david.conlon@maths.ox.ac.uk}}. Research supported by a Royal Society University Research Fellowship and by ERC Starting Grant 676632.} 
 \and Nina Kam\v{c}ev\thanks{Department of Mathematics, ETH, 8092 Z\"urich, Switzerland. Email: \href{mailto:nina.kamcev@math.ethz.ch} {\nolinkurl{nina.kamcev@math.ethz.ch}}.}}
 \date{}

\begin{document}
    \maketitle
    
    \begin{abstract}
    The Hales--Jewett theorem states that for any $m$ and $r$ there exists an $n$ such that any $r$-colouring of the elements of $[m]^n$ contains a monochromatic combinatorial line. We study the structure of the wildcard set $S \subseteq [n]$ which determines this monochromatic line, showing that when $r$ is odd there are $r$-colourings of $[3]^n$ where the wildcard set of a monochromatic line cannot be the union of fewer than $r$ intervals. This is tight, as for $n$ sufficiently large there are always monochromatic lines whose wildcard set is the union of at most $r$ intervals.
    \end{abstract}
    
    \section{Introduction} \label{sec:intro}
    
        The Hales--Jewett theorem~\cite{HJ63} is one of the central results in Ramsey theory. Quoting Graham, Rothschild and Spencer~\cite{GRS90}, it ``strips van der Waerden's theorem of its unessential elements and reveals the heart of Ramsey theory. It provides a focal point from which many results can be derived and acts as a cornerstone for much of the more advanced work."
        
        Stating the theorem requires some notation. Given natural numbers $m$ and $n$, let $[m]^n$ be the collection of all $n$-letter words, where each letter is taken from the alphabet $[m] = \{1, 2, ..., m\}$. Given a word $w$ from $[m]^n$, a subset $S$ of $[n]$ and an element $i$ of $[m]$, let $w(S,i)$ be the word obtained from $w$ by replacing the $j$th letter with $i$ for all $j$ in $S$. A \emph{combinatorial line} in $[m]^n$  with \emph{wildcard set} $S \neq \emptyset$ is then a subset of the form $\left \{w(S, 1), w(S, 2), \dots, w(S, m)\right \}$. 
        
        \begin{HJ}
For any natural numbers $m$ and $r$, there exists a natural number $n$ such that any $r$-colouring of the elements of $[m]^n$ contains a monochromatic combinatorial line.
\end{HJ}

        For $m = 2$, the Hales--Jewett theorem is simple to prove. Consider all sequences of length $r$ of the form $11\dots 122\dots 2$, that is, a string of $1$s followed by a string of $2$s. Since there are $r+1$ different sequences, the pigeonhole principle implies that two of them must receive the same colour. If the first of these sequences switches from $1$s to $2$s after the $i$th letter and the second switches after the $j$th letter with $j > i$, then these two sequences form a monochromatic combinatorial line whose wildcard set is the interval $[i+1, j]$.
    
    Given a word $w$ from $[m]^n$, disjoint subsets $S_1, \dots, S_q$ of $[n]$ and elements $i_1, \dots, i_q$ of $[m]$, let $w(S_1,i_1; \dots; S_q, i_q)$ be the word obtained from $w$ by replacing the $j$th letter with $i_k$ if $j$ is in $S_k$. For $m = 3$, the first step in Shelah's celebrated proof of the Hales--Jewett theorem~\cite{S88} is to show that for $n$ sufficiently large there is a word $w \in [3]^n$ and disjoint intervals $S_1, \dots, S_r$ of $[n]$ such that for any $T \subseteq [r]$ and any $i_1, \dots, i_r \in [3]$, the word $w(S_1, j'_1; \dots; S_r, j'_r)$ obtained by letting $j'_t = 2$ for all $t \in T$ and $j'_t = i_t$ for all $t \notin T$ has the same colour as the word $w(S_1, j''_1; \dots; S_r, j''_r)$ defined analogously by letting $j''_t = 3$ for all $t \in T$ and $j''_t = i_t$ for all $t \notin T$. That is, regardless of how the intervals $S_1, \dots, S_r$ are filled, we may switch the label on any subset of the intervals from $2$ to $3$ without changing the colour of the word. 
    
    To complete the proof, we consider the $r$-colouring of $[2]^r$ where the word $v = v(1) \dots v(r)$ receives the colour of the word $w(S_1, v(1); \dots; S_r, v(r))$. By the $m = 2$ case of the theorem, there is a monochromatic combinatorial line determined by a wildcard set $T \subseteq [r]$. This implies that there are $i_1, \dots, i_r \in [2]$ such that the word $w(S_1, j_1; \dots; S_r, j_r)$ with $j_t = 1$ for all $t \in T$ and $j_t = i_t$ for all $t \notin T$ has the same colour as the word $w(S_1, j'_1; \dots; S_r, j'_r)$ with $j'_t = 2$ for all $t \in T$ and $j'_t = i_t$ for all $t \notin T$. But we already know that this latter word has the same colour as the word $w(S_1, j''_1; \dots; S_r, j''_r)$ with $j''_t = 3$ for all $t \in T$ and $j''_t = i_t$ for all $t \notin T$. Therefore, we have a monochromatic combinatorial line with wildcard set $S = \cup_{t \in T} S_t$.
    
    In particular, this proof shows that it is possible to find monochromatic combinatorial lines in $[3]^n$ where the wildcard set has a comparatively simple structure - it is the union of at most $r$ intervals. The main result of this note says that there are situations where one can do no better, suggesting that the proof strategy described above is, at least in some sense, necessary. 
         
    \begin{THM} For any $n$ and any odd $r>1$, there is an $r$-colouring of $[3]^n$ containing no monochromatic combinatorial line whose wildcard set is the union of fewer than $r$ intervals.
    \end{THM}

    \section{The proof} \label{sec:proof}

    Write $\mathbb{Z}_r$ for the cyclic group with $r$ elements. Fix a vector $t = (t_1, t_2, t_3) \in \mathbb{Z}_r^3$ and, for a word $w \in [3]^n$, let $T'(w) = \sum_{j \in [n]} t_{w(j)}$. In words, $t$ assigns a weight to each letter in $[3]$ and $T'(w)$ is then the sum of the weights over all letters of $w$, where the sum is taken modulo $r$. Let the word $\overline{w}$ be obtained from $w$ by contracting each interval on which $w$ is constant to a single letter. Set $T(w) = T'(\overline{w})$. Finally, we construct the word $w^+$ by inserting a letter 1 at the start and end of $w$. Our colouring of $[3]^n$ will be $T^+(w) = T(w^+)$. For example, for $w= 11122133$, we have $w^+ = 1111221331$,  $\overline{w^+} = 12131$, and $T^+(w)= T(w^+)= t_1+t_2+t_1+t_3+t_1 $. We claim that for $t_1 = t_3 = 2$ and $t_2 = -1$, the colouring $T^+: [3]^n \to \mathbb{Z}_r$ contains no monochromatic combinatorial line whose wildcard set is the union of fewer than $r$ intervals. 
    
	Let us introduce some more notation. 
	Consider a combinatorial line $(x_1, x_2, x_3)$ with $x_i = w(S,i)$, where $S$ is a union of $q$ disjoint non-consecutive intervals in $[n]$. Although any word that coincides with $x_i$ outside the set $S$ can be chosen as the representative $w$, let us set $w = x_1$ to avoid ambiguity. Outside $S$, the word $w^+$ consists of a collection of non-empty subwords $w_0,\, w_1, \dots, w_{q}$, where $w_{j-1}$ precedes $w_j$ for all $j = 1, \dots, q$. In other words, $w_j$ is the subword of $w^+$ lying between the $j^{th}$ and $(j+1)^{th}$ intervals in $S$. Note that the subwords $w_0$ and $w_q$ are non-empty by the construction of $w^+$.  Denote the first letter of $w_j$ by $f_j$ and the last letter by $\ell_{j+1}$ (such an indexing will be more convenient below).
    We now show that the difference $T(x_i^+) - (T(w_0) + \dots + T(w_{q}))$ depends only on $i$ and the letters $\ell_j, f_j$. 
 
    \begin{CLAIM}  For any $t_1, t_2, t_3 \in \mathbb{Z}_r$ and $i \in [3]$,
    \begin{equation} \label{eq:Tchange}
        T^+(x_i) = T(x_i^+) =  T(w_0) + h_i(\ell_1, f_1)+ T(w_1) + h_i(\ell_2, f_2) + \dots + h_i(\ell_{q}, f_{q}) + T(w_{q}),
        \end{equation} where the $h_i(\ell, f)$ are $\mathbb{Z}_r$-valued functions  specified in the proof.
    \end{CLAIM}
    
    \begin{proof}
        We write $y_i = x_i^+$ for $i \in [3]$ so that the identity \eqref{eq:Tchange} is just a statement about how $T$ can be computed from the decomposition of the word $y_i$. Let us first take $q= 1$. There are essentially two cases. Firstly, suppose $\ell_1 \neq f_1$. For concreteness, we take $\ell_1 = 1$, $f_1 = 2$. Then $\overline{y_1} = \overline{y_2} = \overline{w_0} \,\overline{ w_1}$ and, therefore, $T(y_1) = T(y_2) = T(w_0) + T(w_1)$ and $h_1(1, 2) = h_2(1, 2)  = 0$. Moreover, $h_3(1, 2) = t_3$. 
        Suppose now that $\ell_1 = f_1$ and consider the special case $\ell_1 = f_1= 1$. For a word $u$ ending in $1$, let $u\setminus 1$ be the word obtained from $u$ by removing its final letter. Then $\overline{y_1} = \left(\overline{w_0}\setminus 1\right)\, \overline{w_1}$. Hence, $h_1 (1, 1) = T(y_1)-T(w_0)-T(w_1) = -t_1$. Moreover, $h_i(1, 1) = t_i$ for $i = 2, 3$. Since $h_i\left(\ell, f\right) = h_i\left(f, \ell \right)$,  all possible cases are summarised in the following table:
    \begin{center}
        \begin{tabular}{c|cccccc}
        \hline
            $(\ell, f )$ & $(1, 1)$ & $(2, 2)$ & $(3, 3)$ & $(2, 3 )$  & $(3,1 )$ &  $(1,2 )$ \\
        \hline 
            $h_1(\ell,f)$  &  $-t_1$ & $t_1$   & $t_1$   & $t_1$   & $0$ &   $0$\\
           $h_2(\ell, f)$ &  $t_2$  & $-t_2$   & $t_2$   & $0$    & $t_2$ &   $0$\\
           $h_3(\ell, f)$ &   $t_3$ & $t_3$   & $-t_3$   & $0$   & $0$ &   $t_3$\\
           \hline
        \end{tabular}
    \end{center}
    
    The general case now follows by a simple induction. Indeed, suppose that \eqref{eq:Tchange} holds for $q - 1$. We will verify that it also holds for $q$. By the $q = 1$ case discussed above, 
    \begin{equation} \label{eq:induction}
         T(y_i) = T(w_0) + h_i(l_1, f_1) + T(w_1 i w_2  \dots i w_q).
    \end{equation}     
   Since  $w_1$ is non-empty, we can apply the induction hypothesis to the term $T(w_1 i w_2  \dots i w_q)$, which completes the proof. A careful reader may notice that in \eqref{eq:induction}, the intervals of the wildcard set $S$ have been replaced by a single letter $i$, which just facilitates the notation and makes no difference since $T$ is computed from a contraction of the word.
    \end{proof}
        
       	Suppose now that $t_1= t_3 = 2$, $t_2 = -1$ and there is a combinatorial line $(x_1, x_2, x_3)$ such that $T^+(x_1) = T^+(x_2) = T^+(x_3)$. Then, for $i \in \{ 1, 3\}$,  
	$$0 = T^+(x_i) - T^+(x_2) = \sum_{j=1}^q \left( h_i(\ell_j, f_j) - h_2(\ell_j, f_j) \right ).$$ 
	In particular, summing these two equalities, 
	$$0 = T^+(x_1)+T^+(x_3)- 2T^+(x_2) = \sum_{j=1}^q  \left( h_1(\ell_j, f_j) + h_3(\ell_j, f_j) - 2h_2(\ell_j, f_j) \right ).$$
But we can verify that with $t_1 = t_3 = -2t_2 = 2$, for each $\ell$ and $f$ we have $h_1(\ell, f) + h_3(\ell, f) - 2h_2(\ell, f) = 2$, so $ 0 = T^+(x_1)+T^+(x_3)- 2T^+(x_2) = 2q$. 
Since $r$ is odd, $2q = 0$ in $\mathbb{Z}_r$ implies $q \geq r$, as required.

\section{Further remarks} \label{sec:remarks}

For even $r$, our result implies that for any $n$ there is an $r$-colouring of $[3]^n$ containing no monochromatic combinatorial line whose wildcard set is the union of fewer than $r-1$ intervals. It remains to decide whether this can be improved to $r$ intervals. In an earlier version of this paper, we explicitly conjectured that this was possible when $r = 2$. Surprisingly, this conjecture was shown to be false by Leader and R\"aty~\cite{LR18}. That is, for $n$ sufficiently large, every two-colouring of $[3]^n$ contains a monochromatic combinatorial line whose wildcard set is a single interval. It would be very interesting to extend this result to all even $r$.

Given $m$ and $r$, let HJ$(m,r)$ be the smallest dimension $n$ such that every $r$-colouring of $[m]^n$ contains a monochromatic combinatorial line. By following Shelah's proof of the Hales--Jewett theorem~\cite{S88}, one can show that for $n$ sufficiently large depending on $m$ and $r$ there is a monochromatic combinatorial line in $[m]^n$ whose wildcard set is the union of at most HJ$(m-1,r)$ intervals. In our earlier draft, we conjectured that this was always best possible. Unfortunately, the result of Leader and R\"aty shows that this conjecture fails already for $m = 3$ and $r = 2$. As such, we have decided to retract our statement. However, even a marginal improvement on either the upper or lower bound would be interesting.

\begin{QUE}
Do there exist $m \geq 4$, $r \geq 2$ and $c > 1$ such that there are $r$-colourings of $[m]^n$ containing no monochromatic combinatorial line whose wildcard set is the union of at most $c r$ intervals?\\
\end{QUE}

\vspace{-8mm}

\begin{QUE}
For which $m \geq 3$ and $r \geq 2$ is it true that for $n$ sufficiently large, any $r$-colouring of $[m]^n$ contains a monochromatic combinatorial line whose wildcard set is the union of at most $\mathrm{HJ}(m-1, r)-1$ intervals?
\end{QUE}

\vspace{2mm}
\noindent \textbf{Acknowledgement.} The second author would like to thank Sarah Gales, Hannah and Sven Eggimann for hosting her in Oxford while this research was conducted. We would also like to thank the anonymous referee for a number of helpful remarks.

\end{document}